\documentclass[12pt]{article}

\title{The genus of projective curves on complete intersection surfaces}
\author{Rebecca Tramel}
\date{}
\usepackage{amsmath}
\usepackage{amsthm}
\usepackage{amssymb}
\usepackage{verbatim}
\usepackage{tikz}
\usetikzlibrary{calc}
\usetikzlibrary{matrix}
\usetikzlibrary{arrows}
\theoremstyle{definiton}
\newtheorem{thm}{Theorem}[section]

\newtheorem{prop}[thm]{Proposition}
\newtheorem{cor}[thm]{Corollary}
\newtheorem{lem}[thm]{Lemma}

\newtheorem{conj}[thm]{Conjecture}
\newtheorem{ex}[thm]{Example}

\newcommand{\BQ}{\mathbb{Q}}

\newcommand{\BP}{\mathbb{P}}

\newcommand{\CO}{\mathcal{O}}
\newcommand{\Cd}{\mathcal{D}}
\newcommand{\CA}{\mathcal{A}}
\newcommand{\CB}{\mathcal{B}}
\newcommand{\CI}{\mathcal{I}}

\newcommand{\ch}{\rm ch}
\newcommand{\Coh}{\rm Coh}
\newcommand{\Dim}{{\rm dim}}

\begin{document}

\maketitle

\begin{abstract}
We bound the genus of a projective curve lying on a complete intersection surface in terms of its degree and the degrees of the defining equations of the surface on which it lies. 
\end{abstract}

\section{Introduction}

It is natural to consider under what conditions the classical Castelnuovo bounds on the genus of smooth projective curves can be improved upon. We consider curves lying on complete intersection surfaces. Our goal is to relate the degrees of the curve and the surface to the genus of the curve. This is a generalization of the results of \cite{Harris} for curves lying on surfaces in $\BP^3$. We are able to give the following bound on the genus $g$ of such a curve, in terms of its degree, $d$ and the degrees $k_1, \dots, k_{n-2}$ of the defining equations of the surface, when $d$ is large with respect to the degree of the surface. Specifically, we will require the following:
\begin{equation}
\label{ineq}
d\geq k_1\cdots k_{n-2}(k_1+\cdots k_{n-2}).
\end{equation}

\newtheorem*{thm:bound}{Theorem \ref{thm:bound}}
\begin{thm:bound}
Assume $S$ is a complete intersection surface in $\BP^n$ defined by equations of degrees $k_1, \dots, k_{n-2}$, and $C$ is a degree $d$ curve lying on $S$. Suppose the degrees $d,k_1,\dots, k_{n-2}$ satisfy (\ref{ineq}). Let $\epsilon=d-k_1\cdots k_{n-2}\lceil\frac{d}{k_1\cdots k_{n-2}} \rceil$. Then the genus of $C$ is bounded as follows:
$$
g(C) \leq \frac{d^2}{2k_1\cdots k_{n-2}}+\frac{1}{2}d(k_1+\cdots+k_{n-2}-n-1)+p(k_1,\dots, k_{n-2}, \epsilon)
$$ 
where $p(k_1, \dots, k_{n-2}, \epsilon)$ is a polynomial in $k_1, \dots, k_{n-2}, \epsilon$ given explicitly later.
\end{thm:bound}

Considered as a polynomial in $d$, the leading and linear term are sharp. The constant term, given later, is not sharp. It is a degree $n$ polynomial in the degrees $k_i$ of the defining equations of the surface.

\subsection{History of the question}

Our strategy is to bound the genus of the curve $C$ by computing hilbert functions of twists the ideal sheaf of the set of intersection points of $C$ with a general hyperplane. This strategy is used by Castelnuovo to achieve the classical results. Requiring that $C$ lies on a given complete intersection surface allows us to compute these Hilbert functions directly in some cases, as in \cite{Harris}. The main new strategy we use is to compute bounds on the Hilbert functions using a torus degeneration (see Proposition \ref{nonincpropn}).

The question of bounding the genus of projective curves is also addressed in \cite{HarrisEis} and in \cite{CCD}. Both papers address the Halphen problem, of bounding a curve in terms of the smallest degree $s$ of a surface on which a curve lies.  \cite{HarrisEis} gives a bound in terms of the degree $d$ of a curve in $\BP^r$ in terms of $d$, $r$ and $s$ when $d$ is sufficiently large with respect to $s$ and $s$ is not large with respect to $r$. \cite{CCD} extends this result to give a bound in terms of $d$, $s$ and $r$ removing the assumption about the relative sizes of $r$ and $s$. \ These papers are able to give bounds which are sharp. Our results differ, in that we require this surface is a complete intersection surface, and we make weaker assumptions about the degree of this surface. However, in the case that the smallest degree surface is in fact a complete intersection surface satisfying the degree requirements, our bounds agree in the highest term, and our bound improves upon the bound given in \cite{CCD} in the linear term. 

This question is also considered in \cite{CCD2}, \cite{CCD3} and \cite{DiGennaro}. These papers consider curves satisfying certain flag conditions. In our case, requiring that $C$ lies on a complete intersection surface gives a flag of smooth irreducible projective varieties $C \supseteq V_2 \supseteq \cdots \supseteq V_{n-1}$ where $V_i=Z(f_1, \dots, f_{n-i})$. It follows from Bezout's theorem, and our assumption that the degree of $C$ is large, that the curve $C$ cannot lie on any surface of degree less than the degree of $S=V_2$. Repeating this argument inductively, $C$ does not lie on any $i$-dimensional varieties of degree smaller than the degree of $V_i$. 

In the case when $n=4$, \cite{CCD2} gives a sharp bound for curves on such a flag when $d>{\rm max}\{12(k_1k_2)^2, (k_1k_2)^3\}$. \cite[Theorem 2.2]{CCD3} gives a bound for $n \geq 3$ when the $d>(k_1\cdots k_{n-2})^2$ and $k_{n-2} \gg k_{n-3} \gg \cdots \gg k_1$. This bound matches ours in the quadratic term. We improve upon the linear term given in this bound, and require only (\ref{ineq}), with no requirement on the relative sizes of the degrees $k_i$. \cite[Inequality (2.1)]{DiGennaro} gives a bound for the genus of a curve lying on an irreducible surface without requiring it be a complete intersection surface. This result requires that $d>(k_1\cdots k_{n-2})^2-(k_1\cdots k_{n-2})$, which is stronger than our assumption. The bound matches our quadratic and linear terms, while our constant term is a lower degree polynomial in the degree of the surface. \cite{DF} refines this result in the case where the degree of the surface is small with respect to $n$.

\subsection{Additional motivation}

Our motivation for considering the genus of these curves is the study of Bridgeland stability on threefolds. \cite{BMT} conjectures a Bogomolov-Gieseker type inequality for stable sheaves on projective threefolds. This inequality predicts the existence of a genus bound for curves lying on complete intersection threefolds in terms of the degree of the curves, and the degrees of the defining equations of the threefold. In the final section of this paper, we show how the result of Theorem \ref{thm:bound} could give such a bound, if it were extended to curves of low degree.

\section{Curves on complete intersection surfaces}

\subsection{Hilbert functions and the genus of $C$}

Our goal is to compute a bound on the genus of a curve lying on a complete intersection surface in $\BP^n$. Our strategy will be to compute hilbert functions of twists of a particular ideal sheaf, the ideal sheaf of points of intersection of the curve and a general hyperplane in $\BP^n$. We will then use Riemann-Roch to arrive at a bound for the genus of the curve. This strategy follows ideas from \cite{Harris}, where he computed a bound for the case $n=3$. 

Let $C$ be a curve of degree $d$ in $\BP^n$. Suppose $C$ lies on a complete intersection surface $S$ in $\BP^n$. The Riemann-Roch theorem implies that if $g$ is the genus of $C$, for $l\gg 0$,
$$g=dl-h^0(C,\CO(l))+1.$$

We define $\alpha_l$ to be the dimension of the image of the restriction map $$\rho_l \colon H^0(\BP^n,\CO(l))\rightarrow H^0(C,\CO(l)).$$ Then for $l\gg 0$, Riemann-Roch gives
\begin{equation}
\label{eqn}
g = dl-\alpha_l+1.
\end{equation}

Choose $H$ to be a generic hyperplane in $\BP^n$ and $\Gamma=H \cap C$. Let $\CI_{\Gamma}$ be the ideal sheaf of $\Gamma$ in $\BP^n$. Let $$\sigma_l \colon H^0(\BP^n,\CI_{\Gamma}(l))\rightarrow H^0(C,\CI_{\Gamma}(l)\vert_C)$$ the the restriction map. The spaces $H^0(C, \CO(l-1))$ and $H^0(C,\CI_{\Gamma}(l)\vert_C)$ are isomorphic via multiplication by the defining equation of $H$. Further, $H^0(\BP^n,\CO(l-1))$ injects into $H^0(\BP^n,\CI_{\Gamma}(l))$ via the same multiplication map, call it $h$.
\begin{center}
\begin{tikzpicture}
  \matrix (m) [matrix of math nodes,row sep=3em,column sep=4em,minimum width=2em]
  {
     H^0(\BP^n,\CO(l-1)) & H^0(C,\CO(l-1)) \\
    H^0(\BP^n,\CI_{\Gamma}(l)) & H^0(C,\CI_{\Gamma}(l)\vert_C) \\};
  \path[-stealth]
    (m-1-1)  edge node [above] {$\rho_{l-1}$} (m-1-2)
    (m-1-1) edge[right hook->] node [left] {$h$} (m-2-1)
    (m-2-1) edge node [above] {$\sigma_l$} (m-2-2)
    (m-1-2) edge node [below,rotate=90] {$\sim$} (m-2-2)
    (m-1-2) edge node [left] {$h\vert_C$} (m-2-2);
\end{tikzpicture}
\end{center}
 Thus the image of $\rho_{l-1}$ is contained in the the image of $h\vert_C^{-1} \circ \sigma_l$. In other words, $\alpha_{l-1} \leq$ dim Im$\sigma_l$. Then difference $\alpha_l-\alpha_{l-1}$ is bounded from below by the difference $\alpha_l-$ dim Im $\sigma_l$. The kernel of $\rho_l$ is $H^0(\BP^n,\CI_C(l))$, the sections of $\CO(l)$ vanishing on $C$. This is also the kernel of $\sigma_l$. Thus
\begin{align*}
\alpha_l-\alpha_{l-1} &\geq \alpha_l-{\rm dim} \ {\rm Im} \sigma_l \\
&=\left( h^0(\BP^n,\CO(l))-h^0(\BP^n,\CI_C(l))\right)-\left(h^0(\BP^n,\CI_{\Gamma}(l))-h^0(\BP^n,\CI_C(l)) \right) \\
&= h^0(\BP^n,\CO(l))-h^0(\BP^n,\CI_{\Gamma}(l)). \\
\end{align*}
We define $\beta_l$ to be $h^0(\BP^n,\CO(l))-h^0(\BP^n, \CI_{\Gamma}(l))$. 

The sequence of sheaves $\CO_{\BP^n}(l-1) \rightarrow \CO_{\BP^n}(l)\rightarrow \CO_H(l)$ is exact. Further, the sequence $\CO_{\BP^N}(l-1) \rightarrow \CI_{\Gamma}(l)\rightarrow \CI_{\Gamma}(l)\vert_H$ is exact. The restriction map $H^0(\BP^n,\CO_{\BP^n}(l))\rightarrow H^0(H,\CO_H(l)$ is surjective, as is the restriction map $H^0(\BP^n,\CI_{\Gamma}(l))\rightarrow H^0(H,\CI_{\Gamma}(l)\vert_H)$. Taking cohomology of both, we see that 
\begin{align*}
\beta_l &= h^0(\BP^n,\CO(l))-h^0(\BP^n,\CI_{\Gamma}(l))\\
 &= h^0(H,\CO_H(l))+h^0(\BP^n,\CO_H(l-1))-h^0(H,\CI_{\Gamma}(l)\vert_H)-h^0(\BP^n,\CO_H(l-1)) \\
&=h^0(H,\CO(l))-h^0(H ,\CI_{\Gamma}(l)\vert_H ). \\
\end{align*}

Define $\gamma_0:= \beta_0$ and $\gamma_l:=\beta_l-\beta_{l-1}$ for $l \geq 1$. Let $P$ be a generic hyperplane in $H\cong \BP^{n-1}$ containing no points of $\Gamma$. Then the following sequence is exact: 
$$0\rightarrow H^0(H,\CI_{\Gamma}(l-1)\vert_H)\rightarrow H^0(H,\CI_{\Gamma}(l)\vert_H)\rightarrow H^0(P,\CO_P(l)).$$
Define $e_l$ to be the dimension of the image of the second map. In other words, $e_l:=h^0(H,\CI_{\Gamma}(l)\vert_H)-h^0(H,\CI_{\Gamma}(l-1)\vert_H)$. This leads to the following relationship between $\gamma_l$ and $e_l$: $$\gamma_l=\binom{l+n-2}{n-2}-e_l.$$

If we can compute $\gamma_l$, this will allow us to bound the genus of $C$ as follows. First, we claim that
\begin{equation} 
\label{eqn2}
\sum_{i=0}^l\beta_i=\sum_{i=0}^l(l-i+1)\gamma_i.
\end{equation}
 This follows from the fact that $\beta_i=\sum_{k=0}^i \gamma_k$. Further, there is an exact sequence of sheaves
$$0 \rightarrow \CI_{\Gamma}(l) \rightarrow \CO(l) \rightarrow \CO_{\Gamma} \rightarrow 0.$$
For $l \gg 0$, $H^1(\BP^n,\CI_{\Gamma}(l))=0$, and so $\sum_{i=0}^l \gamma_i=d$. Substituting this into (\ref{eqn2}) gives $\sum_{i=0}^l \beta_i=ld-\sum_{i=0}^l(i-1)\gamma_i$. Since $\alpha_l \geq \sum_{i=0}^l\beta_i$, then (\ref{eqn}) gives the following bound on $g$ in terms of $\gamma_i$. 

\begin{lem} 
\label{genuslem}
For $l\gg 0$, $$g \leq \sum_{i=0}^l(i-1)\gamma_i+1.$$
\end{lem}

Our strategy will be to find contraints on the $\gamma_i$, and then to compute a function $\gamma_i^{max}$ which maximizes the right hand side subject to these constraints. Substituting in $\gamma_i^{max}$ to the formula above will give a bound on the genus of any such curve.

\subsection{Calculating $\gamma_i$ for curves of large degree}

Let $C$ be a curve of degree $d$ and genus $g$ lying on complete intersection surface $S$ in $\BP^n$. Say $S$ is defined by equations $f_1, \dots, f_{n-2}$ of degrees $k_1, \dots, k_{n-2}$ respectively. Assume $k_1 \leq k_2 \leq \cdots \leq k_{n-2}$. We will assume that $d$ is large in relation to the degree of $S$, specifically, that (\ref{ineq}) holds.

Let $H \cong \BP^{n-1}$ and $P \cong \BP^{n-2}$ be positioned as in the previous section, with $\Gamma= C \cap H$ a set of $d$ distinct points, and $P \subset H$ containing none of these. For small values of $i$, $\gamma_i$ does not depend on $C$ but only on $S$, and can be computed directly. 

\begin{lem}
\label{smalllemn}
Let $m$ be the smallest integer so that $H^0(H,\mathcal{I}_{\Gamma}(m))$ contains a section $s$ not vanishing on $S$. For $i < m$, $\gamma_i=\binom{i+n-2}{n-2}-\Dim(f_1\vert_P, \dots, f_{n-2}\vert_P)^{(i)}$. For $i \geq m$, $\gamma_i \leq \binom{i+n-2}{n-2}-\Dim(f_1\vert_P, \dots, f_{n-2}\vert_P,s\vert_P)^{(i)}$. 
\end{lem}

\begin{proof}
 For $i \leq m$ all sections of $\CI_{\Gamma}(i)$ lie in the ideal of $S$, which we then restrict to $P \cong \BP^{n-2}$ as before, and so $e_i={\rm dim}(f_1\vert_P, \dots, f_{n-2}\vert_P)^{(i)}$. For $i \geq m$, the inclusion of the ideal $(f_1 \vert_P,\dots ,f_{n-2}\vert_P,s\vert_P)$ in the ideal of $\Gamma$ gives the inequality.
\end{proof}

Note that by Bezout's theorem, if $s$ is a section of $\CI_{\Gamma}(m)$ not vanishing on $S$, $S \cap H \cap Z(s)$ must be $0$-dimensional subvariety of $H$ of degree $mk_1\cdots k_{n-2}$. And so our assumption that $\Gamma$ lies in this intersection forces $m \geq m_0$ where $m_0= \lceil \frac{d}{k_1\cdots k_{n-2}}\rceil$.

\begin{lem}
\label{ideallem}
Suppose $g_1, \dots, g_r$ form a regular sequence in $R=k[x_0, \dots, x_n]$ where $k$ is an algebraically closed field, and the degree of $g_i$ is $d_i$. Let $T$ be the multiset of all partial sums of the $d_i$ with elements repeated when a sum is achieved in multiple ways. For $t \in T$ define ${\rm sgn}(t)$ to be $1$ when $t$ is a sum of an even number of degrees $d_i$, and $-1$ otherwise. Then for $l \geq d_1+\cdots + d_r$,  
$${\rm dim}(g_1,\dots, g_r)^{(l)}=\sum_{t \in T} {\rm sgn}(t)\binom{l-t+n}{n}.$$
\end{lem}
\begin{proof}
For $r=1$ we can give a basis for $(g_1)^{(l)}$ by taking all monomials of degree $l-d_1$ and multiplying these by $g_1$. There are $\binom{l-d_1+n}{n}$ such monomials in $x_0, \dots, x_n$, and there are no relations between them, so the above formula holds.

We now proceed by induction on $r$. There is a short exact sequence $$0\rightarrow R/(g_1,\dots, g_{r-1})\rightarrow R/(g_1,\dots, g_{r-1})\rightarrow R/(g_1,\dots ,g_r)\rightarrow 0$$
where the first map is multiplication by $g_r$. This sequence gives the following relation:
$${\rm dim}(g_1,\dots, g_r)^{(l)}=\binom{i+n}{n}+{\rm dim}(g_1,\dots, g_{r-1})^{(l)}-{\rm dim}(g_1,\dots, g_{r-1})^{(l-d_r)}.$$
If the formula holds for $r-1$ then this equation shows it holds for $r$. 
\end{proof}

Now lemma \ref{smalllemn} implies the following about the vanishing of $\gamma_i$.

\begin{cor}
\label{vancor}
Let $m$ be the smallest integer so that $H^0(H,\mathcal{I}_{\Gamma}(m))$ contains a section $s$ not vanishing on $S$. Then $\gamma_i=0$ for $i\geq m+k_1+\cdots +k_{n-2}-n+2$.
\end{cor}

\begin{proof}
We can compute  $\binom{i+n-2}{n-2}-\Dim(f_1, \dots, f_{n-2},s)^{(i)}$ directly using Lemma \ref{ideallem}. Note, however, that for $i \geq k_1+\cdots k_{n-2}+m$, this describes the Hilbert polynomial of a complete intersection variety in $\BP^{n-2}$ defined by equations of degrees $k_1, \dots, k_{n-2},m$. Since there are $n-1$ defining equations, this is the empty set, and so it is $0$.

Now define $t_{max}={\rm sup}\{t \in T \, \vert \, t<k_1+\cdots +k_{n-2}+m-n+2\}$ where $T$ is as in Lemma \ref{ideallem}, the multiset of partial sums of the degrees $k_1, \dots, k_{n-2},m$. For $i>t_m$ we can compute the following bound on $\gamma_i$ in a method similar to that of Lemma \ref{ideallem}. 
$$\gamma_i \leq \sum_{t \in T, \, t \leq t_{max}} {\rm sgn}(t)\binom{i-t+n-2}{n-2}.$$
Using the fact that the sum over all $t$ of these binomials is $0$, we can rewrite this as 
$$\gamma_i \leq  -\sum_{t \in T, \, t > t_{max}} {\rm sgn}(t)\binom{i-t+n-2}{n-2}.$$
For $k_1+\cdots +k_{n-2}+m-n+2 \leq i < k_1+\cdots +k_{n-2}+m$, each binomial above is $0$, and since $\gamma_i$ is nonnegative, $\gamma_i=0$.  
\end{proof}

We will now show that for $i \geq m$, $\gamma_i$ is nonincreasing with $i$. We will first need the following lemma. 

\begin{lem}
\label{surjlemn}
Given a hyperplane section $L \cong \BP^{n-3} \subset P$ for which $L \cap S$ is empty, the map $H^0(H, \CI_{\Gamma}(i)) \rightarrow H^0(L,\CO_L(i))$ is surjective for $i \geq k_1 + \cdots + k_{n-2}$. 
\end{lem}

\begin{proof}
Let $\rho$ be the map from $H^0(H,\CI_{\Gamma}(i)) \rightarrow H^0(L,\CO_L(i))$ restricting sections to $L$. The dimension of the image of $\rho$ is $\binom{i+n-1}{n-1}-h_z(i)$ where $h_Z(i)$ is the hilbert function $h_Z(i)$ of the variety $Z=\Gamma \cap L$. This is by construction a set of $0$ points in $\BP^{n-3}$. For $i$ sufficiently large, $h_Z(i)=0$. Specifically, consider the ideal of $Z'=S \cap L$. This is also an ideal defining $0$ points in $L$, and is contained in the ideal of $\Gamma \cap L$. We have $h_{Z'}(i)=0$ for $i \geq k_1 + \cdots +k_{n-2}$, implying that $h_Z(i)$ vanishes for $i$ in this range as well.  
\end{proof}

\begin{prop}
\label{nonincpropn}
For $i \geq k_1+\cdots +k_{n-2}$, $\gamma_{i+1} \leq \gamma_i$. 
\end{prop}

\begin{proof}
It is equivalent to show that $e_{i+1}-e_1 \geq \binom{i+n-2}{n-3}$. Fix $L \cong \BP^{n-3}$ as in Lemma \ref{surjlemn} so that $P$ has coordinates $x_0, \dots, x_{n-2}$ and $L$ is given by the equation $x_{n-2}=0$ in $P$. Lemma \ref{surjlemn} shows that for every monomial $x_0^{i_0}\cdots x_{n-3}^{i_{n-3}}$ such that $i_0+\cdots i_{n-3} \geq i$ there is a corresponding element $x_0^{i_0}\cdots x_{n-3}^{i_{n-3}}+x_{n-2}g$ in the ideal $H^0(H,\CI_{\Gamma}(i))$. 

Now consider the torus action on $P$ sending $[x_0:\dots : x_{n-2}]$ to $[x_0:\dots : x_{n-3} : tx_{n-2}]$. Letting this act on the ideal of $\Gamma$, the limit ideal will contain $(x_0,\dots, x_{n-3})^{(i)}$. In \cite{Haiman}, the authors show that the multigraded Hilbert scheme is a projective variety. Therefore, under this degeneration of $\CI_{\Gamma}$ to the new ideal $\CI$, we do not change the hilbert function of the ideal. This will imply that $e_i=h^0(H,\CI(i))-h^0(H,\CI(i-1))$.

Give coordinates $x_0, \cdots, x_{n-3}$ to $L$, and coordinates $x_0, \cdots, x_{n-3},y$ to $P$.  This proof will proceed by induction. As a base case, let $n=4$. Then if a monomial $m$ is contained in $\CI{(i)}$, $x_1m$, $x_2m$ and $ym$ will be contained in $\CI{(i+1)}$. Consider the embedding of $\CI{(i)}$ in $\CI{(i+1)}$ mapping $m$ to $x_1m$. We will show that the dimension of $\CI{(i+1)}$ is at least $i+2$ larger than the dimension of $\CI{(i)}$ by finding $i+2$ monomials which are not in the image of this embedding.

By Lemma \ref{surjlemn}, when $i \geq k_1+k_2$, $\CI{(i)}$ contains the monomial $x_0^{i-t}x_1^t$ for each $t=0, \dots, i$.  For each fixed value of $t$, we can find in $\CI{(i)}$ the monomial $x_0^{i-t-r}x_1^ty^r$ for which $r$ is maximal. Then $x_0^{i-t-r}x_1^ty^{r+1}$ is contained in $\CI{(i+1)}$. Since by assumption, $x_0^{i-t-r-1}x_1^ty^{r+1}$ is not in $\CI{(i)}$ we see that the dimension increases by at least $i+1$. Further, since $x_1^i$ is assumed to be in $\CI{(i)}$, $x_1^{i+1}$ is in $\CI{(i+1)}$. This gives one more new monomial in $\CI{(i+1)}$, showing that the dimension increases by at least $i+2$. 

The following picture illustrates this monomial counting for $n=4$.. The filled dots in the first picture in position $(t,r)$ represent elements $x_0^{3-t-r}x_1^ty^r$ of $\CI{(3)}$. The empty dots represents monomials not contained in the ideal. In the second picture, the black circles represent the elements of $\CI{(4)}$ in the image of the embedding of $\CI{(3)}$ into $\CI{(4)}$ via multiplication by $x_0$, and the gray dots represent the so-called new monomials. Note that Lemma \ref{surjlemn} implies that the dots lying on the $t$-axis are filled in the first picture. 

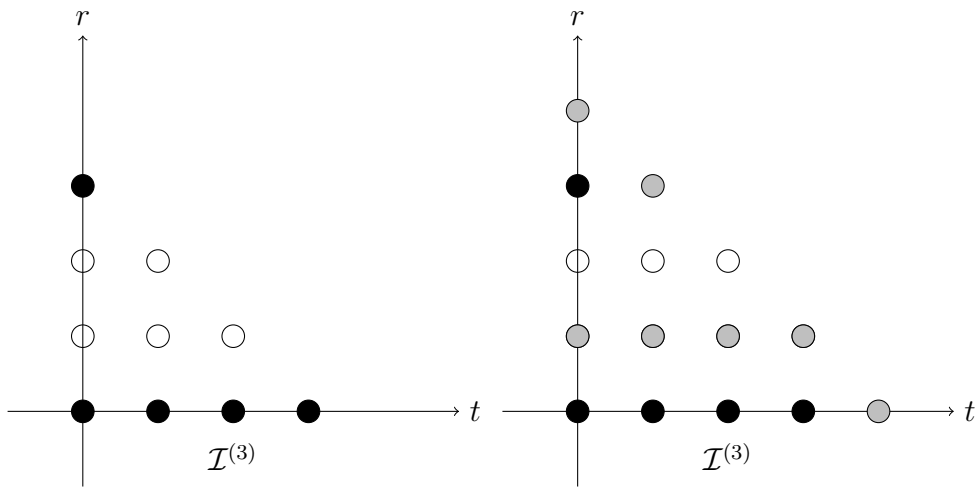
\begin{center}
\begin{tikzpicture}
\draw[->] (-1,0) -- (5,0) node[right] {$t$};
\draw[->] (0,-1) -- (0,5) node[above] {$r$};
\foreach \y in {0,...,3}
\filldraw (\y,0) circle (.15);
\foreach \y in {0,...,2}
\draw (\y,1) circle (.15);
\foreach \y in {0,...,1}
\draw (\y,2) circle (.15);
\filldraw (0,3) circle (.15);
\node at (2,-.6) {$\CI^{(3)}$};
\end{tikzpicture}
\begin{tikzpicture}
\draw[->] (-1,0) -- (5,0) node[right] {$t$};
\draw[->] (0,-1) -- (0,5) node[above] {$r$};
\foreach \y in {0,...,3}
\filldraw (\y,0) circle (.15);
\foreach \y in {0,...,3}
\draw (\y,1) circle (.15);
\foreach \y in {0,...,2}
\draw (\y,2) circle (.15);
\filldraw (0,3) circle (.15);
\node at (2,-.6) {$\CI^{(3)}$};
\filldraw[fill=lightgray] (1,3) circle (.15);
\filldraw[fill=lightgray] (0,4) circle (.15);
\filldraw[fill=lightgray] (4,0) circle (.15);
\foreach \y in {0,...,3}
\filldraw[fill=lightgray] (\y,1) circle (.15);
\end{tikzpicture}
\end{center}

Now we will return to general $n$, and assume the proposition holds for $n-1$. Consider again the injection $\CI^{(i)} \hookrightarrow \CI^{(i+1)}$ sending monomial $m$ to monomial $x_0m$. We will find $\binom{i+n-2}{n-3}$ monomials in $\CI^{(i+1)}$ which cannot be written in the form $x_0m$ for a monomial in $\CI^{(i)}$, showing the dimension has increased by at least this much. By Lemma \ref{surjlemn}, for $i \geq k_1+\cdots +k_{n-2}$, $\CI^{(i)}$ contains all monomials of degree $i$ in the variables $x_0, \dots, x_{n-3}$. For any fixed monomial of this form, $m=x_0^{t_0}\cdots x_{n-3}^{t_{n-3}}$, there is a maximum value of $r$ for which $m_r=x_0^{t_0-r}x_1^{t_1}\cdots x_{n-3}^{t_{n-3}}y^r$ lies in $\CI^{(i)}$. Then $ym_r$ lies in $\CI^{(i+1)}$, but $ym_r/x_0$ does not lie in $\CI(i)$. This strategy gives $\binom{i+n-3}{n-3}$ new monomials, each containing $y$ as a factor.

Now consider all monomials of degree $i$ in $x_0, \dots, x_{n-4}$. Again, by Lemma \ref{surjlemn}, all such monomials are in $\CI^{(i)}$. Then by the inductive hypothesis, we can find $\binom{i+n-3}{n-4}$ new monomials in $\CI^{(i+1)}$ in the variables $x_0, \dots, x_{n-3}$. This gives a total of $\binom{i+n-2}{n-3}$ new monomials, as needed. 
\end{proof}

We now have the following constraints for $\gamma_i$.
\begin{enumerate}
\item For $0 \leq i < m$, $\gamma_i= \binom{i+n-2}{n-2}-\dim (f_1|_P,\dots,f_{n-2}|_P)^{(i)}$.
\item For $i \geq m$, $\gamma_i \leq \binom{i+n-2}{n-2}-\dim(f_1|_P,\dots,f_{n-2}|_P,s\vert_P)^{(i)}$.
\item For $i \geq m$, $\gamma_i$ is non increasing with $i$. 
\item $\sum_{i=0}^\infty \gamma_i=d$.
\end{enumerate}
The following picture illustrates these constraints for $n=4$. $\gamma_i$ must lie along solid lines and below dashed lines. For $n>4$ the picture would be similar, with the graph broken into as many as $2^{n-1}$ pieces, one for each subset of the set $\{k_1, \dots, k_{n-2}, m\}$, and achieving a maximum height of $k_1\cdots k_{n-2}$. 

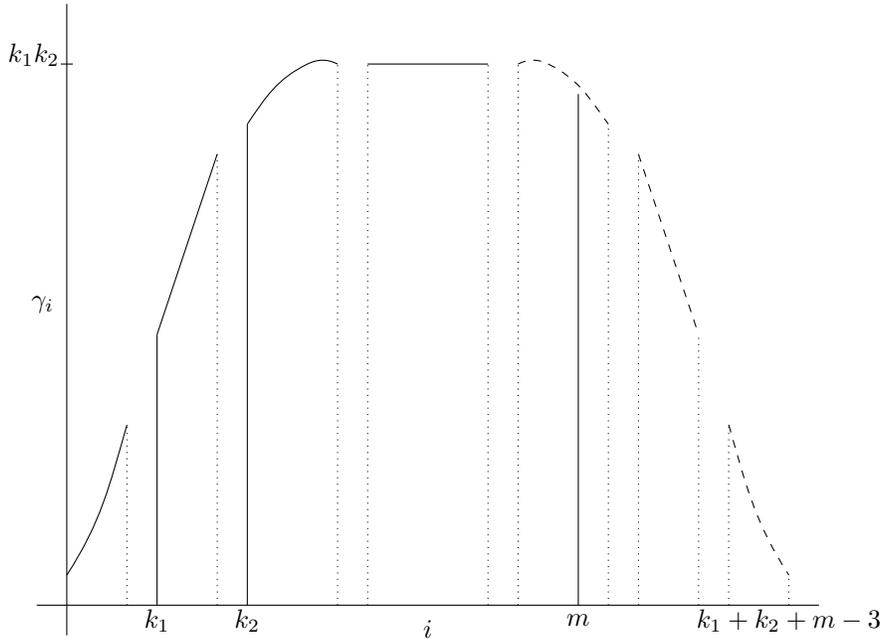
\begin{center}
\begin{tikzpicture}[scale=0.4]
\draw (-1,0) -- (25,0);
\draw (0,-1) -- (0,20);
\draw (0,1) to [out=57, in=254] (2,6);
\draw[dotted] (2,0) -- (2,6);
\draw (3,9) -- (5,15);
\draw (3,0) -- (3,9);
\draw[dotted] (5,0) -- (5,15);
\draw (6,16) to [out=56, in=206] (8,18);
\draw (8,18) to [out=26, in=154] (9,18);
\draw (6,0) -- (6,16);
\draw[dotted] (9,0) -- (9,18);
\draw (10,18) -- (14,18);
\draw[dotted] (10,0) -- (10,18);
\draw[dotted] (14,0) -- (14,18);
\draw[dashed] (15,18) to [out=26,in=154] (16,18);
\draw[dashed] (16,18) to [out=-26,in=126] (18,16);
\draw[dotted] (15,0) -- (15,18);
\draw (17,0) -- (17,17);
\draw[dotted] (18,0) -- (18,16);
\draw[dashed] (19,15) -- (21,9);
\draw[dotted] (19,0) -- (19,15);
\draw[dotted] (21,0) -- (21,9);
\draw[dashed] (22,6) to [out=-74,in=123] (24,1);
\draw[dotted] (22,0) -- (22,6);
\draw[dotted] (24,0) -- (24,1);
\node at (3,-.5) {\footnotesize{$k_1$}};
\node at (6,-.5) {\footnotesize{$k_2$}};
\node at (17,-.5) {\footnotesize{$m$}};
\node at (24,-.5) {\footnotesize{$k_1+k_2+m-3$}};
\node at (12,-.8) {\footnotesize{$i$}};
\node at (-.8,10) {\footnotesize{$\gamma_i$}};
\draw (-.2,18) -- (.2,18);
\node at (-1.1,18.3) {\footnotesize{$k_1k_2$}};
\end{tikzpicture}
\end{center}

Because we know that $\sum_{i=0}^\infty \gamma_i=d$, we can compute $$\sum_{i=m}^\infty \gamma_i=d-mk_1\cdots k_{n-2}+\frac{1}{2}k_1\cdots k_{n-2}(k_1+\cdots +k_{n-2}+n-2).$$  Given a fixed value of $m$, the function $\gamma_{i,m}^{max}$ will be the function $\gamma_i$  satisfying this sum which has area as far to the right as possible. This is drawn below, again in the case $n=4$.
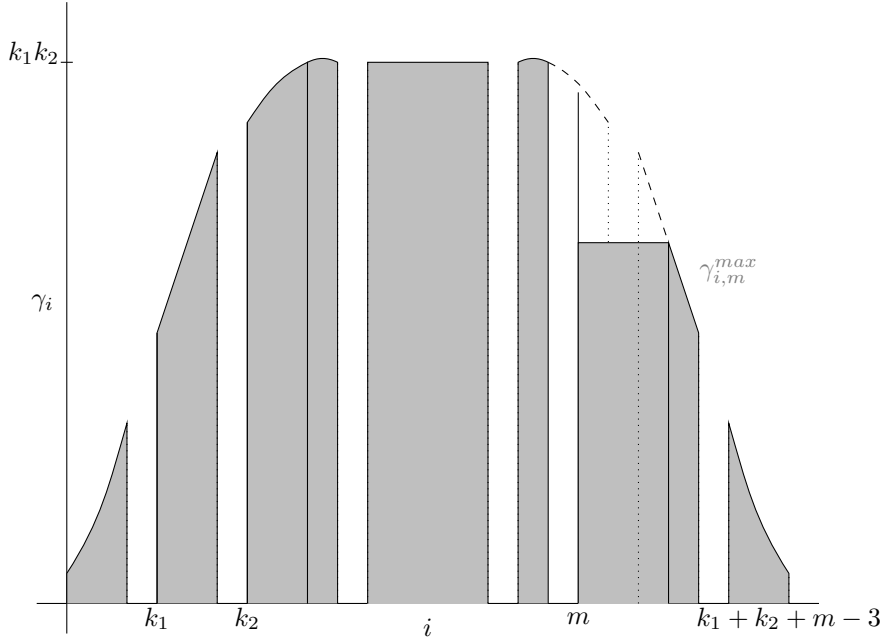
\begin{center}
\begin{tikzpicture}[scale=0.4]
\draw (-1,0) -- (25,0);
\draw (0,-1) -- (0,20);
\draw[fill=lightgray] (0,0) -- (0,1) to [out=57, in=254] (2,6) -- (2,0);
\draw[dotted] (2,0) -- (2,6);
\draw[fill=lightgray] (3,0) -- (3,9) -- (5,15) -- (5,0);
\draw (3,0) -- (3,9);
\draw[dotted] (5,0) -- (5,15);
\draw[fill=lightgray] (6,0) -- (6,16) to [out=56, in=206] (8,18) -- (8,0);
\draw[fill=lightgray] (8,0) -- (8,18) to [out=26, in=154] (9,18) -- (9,0);
\draw (6,0) -- (6,16);
\draw[dotted] (9,0) -- (9,18);
\draw[fill=lightgray] (10,0) -- (10,18) -- (14,18) -- (14,0);
\draw[dotted] (10,0) -- (10,18);
\draw[dotted] (14,0) -- (14,18);
\draw[fill=lightgray] (15,0) -- (15,18) to [out=26,in=154] (16,18) -- (16,0);
\draw[dashed] (16,18) to [out=-26,in=126] (18,16);
\draw[dotted] (15,0) -- (15,18);
\draw (17,0) -- (17,17);
\draw[dotted] (18,0) -- (18,16);
\draw[dashed] (19,15) -- (20,12);
\draw[fill=lightgray] (17,0) -- (17,12) -- (20,12) -- (20,0);
\draw[fill=lightgray] (20,0) -- (20,12) -- (21,9) -- (21,0);
\draw[dotted] (19,0) -- (19,15);
\draw[dotted] (21,0) -- (21,9);
\draw[fill=lightgray] (22,0) -- (22,6) to [out=-74,in=123] (24,1) -- (24,0);
\draw[dotted] (22,0) -- (22,6);
\draw[dotted] (24,0) -- (24,1);
\node at (3,-.5) {\footnotesize{$k_1$}};
\node at (6,-.5) {\footnotesize{$k_2$}};
\node at (17,-.5) {\footnotesize{$m$}};
\node at (24,-.5) {\footnotesize{$k_1+k_2+m-3$}};
\node at (12,-.8) {\footnotesize{$i$}};
\node at (-.8,10) {\footnotesize{$\gamma_i$}};
\draw (-.2,18) -- (.2,18);
\node at (-1.1,18.3) {\footnotesize{$k_1k_2$}};
\node[gray] at (22, 11) {\footnotesize{$\gamma_{i,m}^{max}$}};
\end{tikzpicture}
\end{center}
This function can be calculated for each $m$, and then $\gamma_i^{max}$ is the function $\gamma_{i,m}^{max}$ which maximizes the sum $\sum (i-1)\gamma_{i,m}^{max}$. Then the genus of the curve will be bounded by $\frac{1}{2}k_1\cdots k_{n-2}m(k_1+\cdots +k_{n-2}+m-n-1)+1-C$, where $C$ is the weighted sum of the shaded area in the picture. 

This strategy will easily give a genus bound for $C$ given specific values of $k_i$ and $d$. However, in order to compute a general bound we will relax the second constraint.  For the purpose of this computation, we will replace the constraint that $\gamma_i \leq \binom{i+n-2}{n-2}-\dim(f_1 |_P,,\dots, f_{n-2}|_P,s\vert_P)^{(i)}$ for $i \geq m$ with the less restrictive constraint that $\gamma_i=0$ for $i \geq m+k_1+\cdots +k_{n-2}-n+2$. That is, we will require the following of $\gamma_i$.
\begin{enumerate}
\item For $0 \leq i < m$, $\gamma_i= \binom{i+n-2}{n-2}-\dim (f_1|_P,\dots,f_{n-2}|_P)^{(i)}$.
\item For $i \geq m+k_1+\cdots +k_{n-2}-n+2$, $\gamma_i=0$.
\item For $i \geq m$, $\gamma_i$ is non increasing with $i$. 
\item $\sum_{i=0}^\infty \gamma_i=d$.
\end{enumerate}
The function $\gamma_{i,m}^{max}$ which satisfies these constraints which maximizes $\sum (i-1)\gamma_{i,m}^{max}$ is the function which is constant for $m \leq i < m+k_1+\cdots +k_{n-2}-n+2$ and sums to $d$. 

$$
\gamma_{i,m}^{max} =
\begin{cases}
\binom{i+n-2}{n-2}-\dim(f_1|_P,\dots ,f_{n-2}|_P)^{(i)}, & \text{if } 0 \leq i < m \\
\frac{1}{2}k_1\cdots k_{n-2}+\frac{d-k_1\cdots k_{n-2}m}{k_1+\cdots +k_{n-2}-n+2}, & \text{if } m \leq i <m+k_1+\cdots+k_{n-2}-n+2 \\
\end{cases}
$$
The following picture illustrates $\gamma_{i,m}^{max}$ for $n=4$.

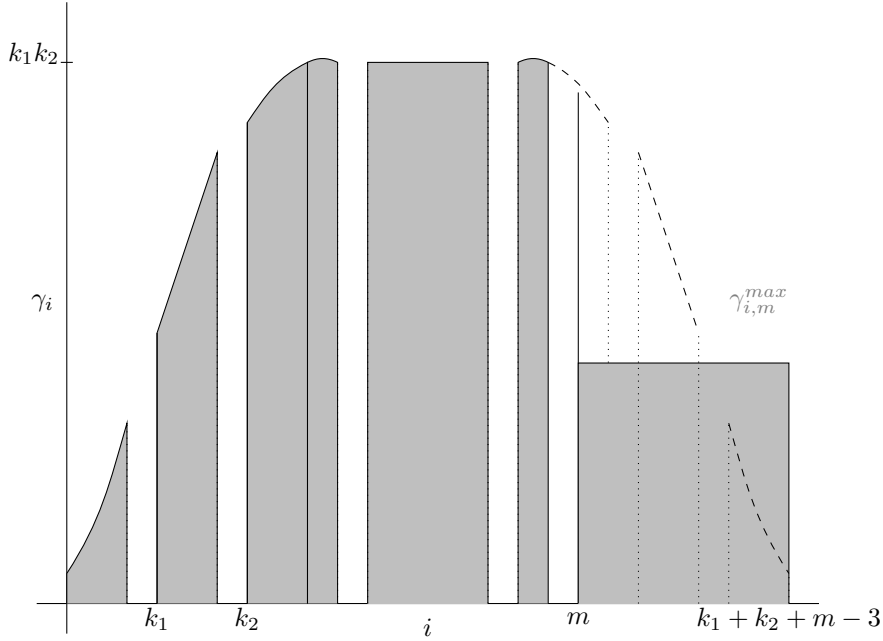
\begin{center}
\begin{tikzpicture}[scale=0.4]
\draw (-1,0) -- (25,0);
\draw (0,-1) -- (0,20);
\draw[fill=lightgray] (0,0) -- (0,1) to [out=57, in=254] (2,6) -- (2,0);
\draw[dotted] (2,0) -- (2,6);
\draw[fill=lightgray] (3,0) -- (3,9) -- (5,15) -- (5,0);
\draw (3,0) -- (3,9);
\draw[dotted] (5,0) -- (5,15);
\draw[fill=lightgray] (6,0) -- (6,16) to [out=56, in=206] (8,18) -- (8,0);
\draw[fill=lightgray] (8,0) -- (8,18) to [out=26, in=154] (9,18) -- (9,0);
\draw (6,0) -- (6,16);
\draw[dotted] (9,0) -- (9,18);
\draw[fill=lightgray] (10,0) -- (10,18) -- (14,18) -- (14,0);
\draw[dotted] (10,0) -- (10,18);
\draw[dotted] (14,0) -- (14,18);
\draw[fill=lightgray] (15,0) -- (15,18) to [out=26,in=154] (16,18) -- (16,0);
\draw[dashed] (16,18) to [out=-26,in=126] (18,16);
\draw[dotted] (15,0) -- (15,18);
\draw (17,0) -- (17,17);
\draw[dotted] (18,0) -- (18,16);
\draw[fill=lightgray] (17,0) -- (17,8) -- (24,8) -- (24,0);
\draw[dashed] (19,15) -- (20,12);
\draw[dashed] (20,12) -- (21,9);
\draw[dotted] (19,0) -- (19,15);
\draw[dotted] (21,0) -- (21,9);
\draw[dashed] (22,6) to [out=-74,in=123] (24,1);
\draw[dotted] (22,0) -- (22,6);
\draw[dotted] (24,0) -- (24,1);
\node at (3,-.5) {\footnotesize{$k_1$}};
\node at (6,-.5) {\footnotesize{$k_2$}};
\node at (17,-.5) {\footnotesize{$m$}};
\node at (24,-.5) {\footnotesize{$k_1+k_2+m-3$}};
\node at (12,-.8) {\footnotesize{$i$}};
\node at (-.8,10) {\footnotesize{$\gamma_i$}};
\draw (-.2,18) -- (.2,18);
\node at (-1.1,18.3) {\footnotesize{$k_1k_2$}};
\node[gray] at (23, 10) {\footnotesize{$\gamma_{i,m}^{max}$}};
\end{tikzpicture}
\end{center}

The sum in Lemma \ref{genuslem} is now simple to compute for $i \geq k_1+\cdots k_{n-2}$, where the function $\gamma_{i,m}^{max}$ is piecewise constant. This makes the following result an easy computation.

\begin{thm}
\label{thm:leading}
The genus $g$ of $C$ is bounded above by a second degree polynomial in $d$ whose leading terms are $\frac{d^2}{2k_1\cdots k_{n-2}}+\frac{d}{2}(k_1+\cdots +k_{n-2}-n-1)$. 
\end{thm}

\begin{proof}
By Lemma \ref{genuslem}, we can compute a bound on $g$ by maximizing $\sum (i-1) \gamma_{i,m}^{max}$ with respect to $m$. Note first that for $i < k_1+\cdots +k_{n-2}$, $\gamma_{i,m}^{max}$ contains no $m$ terms, and so can be ignored in this computation. 

We then compute $\sum_{i \geq k_1+ \cdots k_{n-2}}(i-1)\gamma_{i,m}^{max}$ and ignore any terms not involving $m$ to get $dm-\frac{1}{2}k_1\cdots k_{n-2}m^2$. This is maximized at $\frac{d}{k_1\cdots k_{n-2}}$, and strictly decreasing for $m$ greater than this. As explained before, Bezout's theorem gives a minimum value of $m_0=\lceil \frac{d}{k_1\cdots k_{n-2}} \rceil$ for $m$, and so we set $\gamma_{i}^{max}=\gamma_{i,m_0}^{max}$ This gives the above result.
\end{proof}

Note that these terms agree with the higher order terms for the genus of a complete intersection curve. In that case, the constant term would be $1$. In order to compute our constant term, we need to be able to compute the sum for small $i$ as well. Here, the function $\gamma_{i,m}^{max}$ can be computed using binomial coefficients. We will use the following binomial identity in order to compute this.

\begin{lem}
\label{calclem}
Let $A$ and $B$ be two nonnegative integers. Then
$$\sum_{i=A}^{A+B-1}(i-1)\binom{i+n-2-A}{n-2}=\frac{1}{n}\binom{B+n-2}{n-1}(nA+(n-1)B-2n+1).$$
\end{lem}

\begin{proof}
First we write 
\begin{align*}
\sum_{i=A}^{A+B-1}(i-1)\binom{i+n-2-A}{n-2}&=(A-1)\sum_{i=0}^{B-1}\binom{i+n-2}{n-2}+\sum_{i=0}^{B-1}j\binom{i+n-2}{n-2} \\
&= (A-1)\binom{B+n-2}{n-1}+\sum_{i=0}^{B-1}i\binom{i+n-2}{n-2} . \\
\end{align*}
It is then equivalent to prove that $$\sum_{i=0}^{B-1}i\binom{i+n-2}{n-2}=\frac{1}{n}B(B-1)\binom{B+n-2}{n-2}.$$
We proceed by induction on $B$. For $B=1$ the claim is obvious. Now suppose it holds for $B=k$. Then 
\begin{align*}
\sum_{i=0}^k i\binom{i+n-2}{n-2}&=\frac{1}{n}k(k-1)\binom{k+n-2}{n-2}+k\binom{k+n-2}{n-2} \\
&= \frac{1}{n}k(k+n-1)\binom{k+n-2}{n-2} \\
&= \frac{1}{n}k(k+1)\binom{k+n-1}{n-2}. \\
\end{align*}
\end{proof}

\begin{lem}
\label{lem:stir}
Let $T$ be the multiset of nonzero partial sums of the degrees $k_i$, with an element repeated in $T$ for each way in which it can be formed as a sum of $k_i$. Then
\begin{multline*}
\sum_{t \in T} \frac{(-1)^n}{n}{\rm sgn}(t) t\binom{t+n-2}{n-1}=-k_1\cdots k_{n-2}\left(\frac{1}{6}\sum_{i=1}^{n-2}k_i^2+\frac{1}{4}\sum_{i=2}^{n-2}\sum_{j=1}^ik_ik_j \right. \\ +\left. \frac{\binom{n-1}{2}}{2n}\sum_{i=1}^{n-2}k_i  +\frac{(n-2)!(3n-4)\binom{n-1}{3}}{4n!}\right).
\end{multline*}
\end{lem}
\begin{proof}
All terms not divisible by $k_1\cdots k_{n-2}$ are canceled by later terms in the alternating sum above, so it suffices to determine the coefficients of terms divisible by this monomial in the product $\frac{1}{n}(k_1+\cdots +k_{n-2})\binom{k_1+\cdots k_{n-2}+n-2}{n-1}.$ That is, we need to calculate the coefficients of terms divisible by $k_1\cdots k_{n-2}$ in $\frac{1}{n!}(k_1+\cdots + k_{n-2})^2(k_1+\cdots +k_{n-2}+1)\cdots (k_1+\cdots +k_{n-2}+n-2)$.

For terms of degree $n$, this is the same as the coefficient in $\frac{1}{n!}(k_1+\cdots +k_{n-2})^n$, which can be computed with multinomial coefficients. For the terms of degree $n-1$, the coefficient will be the product of the multinomial coefficient of $k_ik_1\cdots k_{n-2}$ in $(k_1+\cdots k_{n-2})^{n-1}$ multiplied with the sum $1+\cdots +n-2$, and $\frac{1}{n!}$ giving $\frac{\binom{n-1}{2}(n-1)!}{2n!}$. Similarly, the coefficient of the degree $n-2$ term is found as the coefficient of $k_1\cdots k_{n-2}$ in $(k_1+\cdots +k_{n-2})^{n-2}$ multiplied by $\frac{1}{n!}$ and the sum of all products of two numbers in the list $1, \dots, n-2$. 
\end{proof}

We are now able to state the full genus bound for $C$. 

\begin{thm}
\label{thm:bound}
Let $\epsilon=d-k_1\cdots k_{n-2}\lceil\frac{d}{k_1\cdots k_{n-2}} \rceil$. Then the genus of $C$ is bounded as follows:
\begin{multline*} 
g(C) \leq \frac{d^2}{2k_1\cdots k_{n-2}}+\frac{1}{2}d(k_1+\cdots+k_{n-2}-n-1)-\frac{\epsilon^2}{2k_1\cdots k_{n-2}}+1  \\ +\frac{1}{12}k_1\cdots k_{n-2}\left(\sum_{i=1}^{n-2}k_i^2 +3\sum_{i=2}^{n-2}\sum_{j=1}^ik_ik_j-3(n-2)\sum_{i=1}^{n-2}k_i+\frac{(n-2)(3n-5)}{2}\right).
\end{multline*}
\end{thm}

\begin{proof}
The bound is computed using Lemma \ref{genuslem}, and the function $\gamma_{i,m}^{max}$. It then follows from Theorem \ref{thm:leading} and Lemma \ref{calclem} that 
\begin{multline*} 
g(C) \leq\frac{d^2}{2k_1\cdots k_{n-2}}+\frac{1}{2}d(k_1+\cdots+k_{n-2}-n-1)-\frac{\epsilon^2}{2k_1\cdots k_{n-2}}+1  \\ -\frac{1}{4}k_1\cdots k_{n-2}[(k_1+\cdots +k_{n-2})^2+(2n-7)(k_1+\cdots +k_{n-2})-n^2+n]\\ +\sum_{t \in T} (-1)^n{\rm sgn}(t)\frac{1}{n}\binom{t+n-2}{n-1}(nk_1+\cdots +nk_{n-2}-t-2n+1).
\end{multline*}
Then Lemma \ref{lem:stir} gives the constant term. 
\end{proof}

We give this bound explicitly for small $n$, in order to give a sense of how large the constant term becomes. 

\begin{ex}
\label{smalln4}
In $\BP^4$, the genus of a curve $C$ satisfying the above conditions is bounded by the following polynomial:
\begin{multline*}
g(C) \leq \frac{d^2}{2k_1k_2}+\frac{1}{2}d(k_1+k_2-5)-\frac{\epsilon^2}{2k_1k_2}+1 \\ +\frac{1}{12}k_1k_2(k_1^2+k_2^2+3k_1k_2-6(k_1+k_2)+7).
\end{multline*}
\end{ex}

\begin{ex}
\label{smalln5}
In $\BP^5$, the genus of a curve $C$ satisfying the above conditions is bounded by the following polynomial:
\begin{multline}
g(C) \leq \frac{d^2}{2k_1k_2k_3}+\frac{1}{2}d(k_1+k_2+k_3-6)-\frac{\epsilon^2}{2k_1k_2k_3}+1 +\frac{1}{12}k_1k_2k_3(k_1^2+k_2^2+k_3^2 \\ +3k_1k_2+3k_1k_3+3k_2k_3-9(k_1+k_2+k_3)+15).
\end{multline}
\end{ex}

\section{Possible future application}
Theorem \ref{thm:bound} applies to curves of large degree compared with the degree of the surface on which they lie. If this bound could be extended to curves of low degree, then we could hope to apply our result to an open problem in the study of Bridgeland stability on threefolds,  which we now describe. 

Suppose $X$ is a smooth projective threefold. The biggest open question in the study of Bridgeland stability is to define a stability condition on $\Cd^b(X)$, the derived category of coherent sheaves on $X$. In \cite{BMT}, the authors give a conjectured stability condition, which we describe now.

Given an ample class $\omega \in$ NS$_{\BQ}(X)$ and a class $B \in$ NS$_{\BQ}(X)$, we can a heart of a bounded t-structure $\CB_{\omega,B}$ in $\Cd^b(X)$ to be a tilt of $\Coh(X)$ at a slope function depending on $\omega$ and $B$. We can then define a new slope function on $\CB_{\omega,B}$ as follows. For $F \in \CB_{\omega,B}$, 
$$\nu_{\omega,B}(F):=\frac{\omega\ch_2^B(F)-\frac{\omega^3}{6}\ch_0^B(F)}{\omega^2\ch_1^B(F)}.$$
Tilting again by this new slope function, we can define a second heart $\CA_{\omega,B}$ in $\Cd^b(X)$. In \cite{BMT}, the authors expect that the slope function $\nu_{\omega,B}$ defines a stability condition on $\CA_{\omega,B}$. This would follow from the following conjectured Bogomolov-Gieseker type inequality.

\begin{conj}\cite[Conjecture 3.2.7]{BMT}
\label{BMTConj}
For any tilt-stable object $E \in \CB_{\omega, B}$ satisfying 
$$\frac{\omega^3}{6}\ch_0^B(E)=\omega\ch_2^B(E),$$ we have the following inequality:
$$\ch_3^B(E)\leq\frac{\omega^2}{18}\ch_1^B(E).$$
\end{conj}

This conjecture has been proved to hold for specific threefolds. \cite{Macri} shows that \ref{BMTConj} holds for $X=\BP^3$. Furthermore, \cite{Maciocia} shows that \ref{BMTConj} holds for principally polarized abelian varieties under a specific choice for $\omega$ and $B$, and in \cite{Maciocia2} for any choice of $\omega$ and $B$ when the Picard rank is $1$. The inequality was proved for smooth quadric threefolds in \cite{Schmidt}. 

Consider now the special case in which $X$ is a complete intersection three-fold $X=Z(f_1,\dots, f_{n-3})$ in $\BP^n$ where $f_i$ is a homogenous polynomial of degree $k_i$. Suppose now that $C$ is a curve of degree $d$ and genus $g$ lying in $X$. Let $\CI_C$ be the ideal sheaf of $C$ in $X$. 

Let $H$ be a hyperplane section of $X$. There is a unique positive multiple of $H$, call it $\omega$, for which $\CI_C$ is tilt-stable with respect to $\nu_{\omega,0}$. For this value of $\omega$, Conjecture \ref{BMTConj} states that 
$$\ch_3(E) \leq \frac{t^2H^2}{18}\ch_1(E).$$ This simplifies to the following statement as a special case of \ref{BMTConj} (see \cite[Example 7.2.4]{BMT} for the calculation).

\begin{conj}
\label{Genusconj}
If $X$ is a complete intersection threefold as before, and $C$ is a degree $d$ curve lying on $X$ such that $d\leq\frac{1}{2}(k_1\cdots k_{n-3})$, then $$g \leq \frac{d}{2}(k_1+\cdots +k_{n-3}-n-1)+\frac{2d}{3}+1.$$
\end{conj}

This conjectured genus bound is generalized to curves of any degree lying on complete intersection three-folds in  \cite[Section 4]{BMS}. They conjecture the following Catelnuovo inequality:

\begin{conj} \cite[Example 4.4]{BMS}
\label{Genusconjgen}
 $$g(C) \leq \frac{2d^2}{3k_1\cdots k_{n-3}}+\left(\frac{5+3(k_1+\cdots+k_{n-3}-n-1)}{6}\right)d+1.$$
\end{conj} 

Let $X\subset \BP^n$ be a complete intersection threefold as before, and let $C$ be a degree $d$ curve lying on $X$. Consider a generic hyperplane $H$ in $\BP^n$ intersecting $C$ in $d$ distinct point. Define $\Gamma=H \cap C$. Let $m_0$ be the smallest integer so that for a generic choice of $H$, $H^0(H,\CI_{\Gamma}(m_0)) \neq 0$. The following proposition argues that such a curve lies on a complete intersection surface in $\BP^n$.

\begin{prop}
\label{threefoldprop}
The curve $C$ lies on a complete intersection surface in $\BP^n$ defined by equations of degrees $k_1, \dots, k_{n-3}$ and $m_0$.
\end{prop}

\begin{proof}
Choose hyperplanes $H_1=Z(h_1)$ and $H_2=Z(h_2)$, such that $H^0(H_i,\CI_{\Gamma_i}(m_0)) \neq 0$ and such that $P:= H_1 \cap H_2$ does not intersect $C$. Then there is a pencil of hyperplanes intersecting along $P$ in $\BP^n$ given as $H_{\lambda_1,\lambda_2}=Z(\lambda_1 h_1+\lambda_2h_2)$ where $[\lambda_1 : \lambda_2] \in \BP^1$. 

Consider the blow-up Bl$_P\BP^n$ of $\BP^n$ along $P$. Since $C$ does not intersect $P$, $\widetilde{C} \cong C$ lies in Bl$_{P \cap X}X$. Let $\widetilde{H}_{\lambda_1,\lambda_2}$ be the proper transform of $H_{\lambda_1,\lambda_2}$. In Bl$_P\BP^n$, the $\widetilde{H}_{\lambda_1,\lambda_2}$ are disjoint. Further, by assumption there is a nonzero section $s_{\lambda_1,\lambda_2} \in H^0(\widetilde{H}_{\lambda_1,\lambda_2},\CI_{\widetilde{\Gamma}_{\lambda_1,\lambda_2}}(m_0))$ for each $[\lambda_1 : \lambda_2]$ in $\BP^1$. 

Consider the map $p \colon {\rm Bl}_P\BP^n \rightarrow \BP^1$ sending $\widetilde{H}_{\lambda_1,\lambda_2}$ to $[\lambda_1 : \lambda_2]$. For a generic point of $\BP^1$, there is an isomorphism
$$H^0([\lambda_1 : \lambda_2], p_*\CI_{\widetilde{C}}(m_0) \vert_{[\lambda_1 : \lambda_2]} ) \cong H^0(\widetilde{H}_{\lambda_1,\lambda_2},\CI_{\widetilde{\Gamma}_{\lambda_1,\lambda_2}}(m_0)).$$
Thus there is a global section $s_{[\lambda_1,\lambda_2]}$ of $ p_*\CI_{\widetilde{C}} \vert_{[\lambda_1 : \lambda_2]}$ corresponding to $s_{\lambda_1,\lambda_2}$. By Serre vanishing, there is a global section of $p_*\CI_{\widetilde{C}}(l)$ restricting to this $s_{[\lambda_1,\lambda_2]}$ for $l$ sufficiently large.  This section pulls back to a global section $s$ of $\CI_{\widetilde{C}}$ restricting to $s_{\lambda_1,\lambda_2}$ on $\widetilde{H}_{\lambda_1,\lambda_2}$. Set $Y$ to be the irreducible component of $Z(s)$ in Bl$_P\BP^n$ which maps dominantly to $\BP^1$.

Let $\pi$ be the projection map Bl$_P\BP^n \rightarrow \BP^n$. Then $S:=X \cap \pi(Y)$ is a complete intersection surface in $\BP^n$ defined by equations of degrees $k_1, \dots, k_{n-3}$ and $m_0$ containing $C$.
\end{proof}

So long as $m_0$ is not large, so that (\ref{ineq}) holds for the surface defined by equations of degrees $k_1,\cdots, k_{n-3},m_0$, Theorem \ref{thm:bound} and Proposition \ref{threefoldprop} imply together that that the bound in Conjecture \ref{Genusconjgen} holds. In order to show that Conjecture \ref{Genusconjgen} holds for all curves lying on complete intersection threefolds, a weaker bound would need to be shown for curves of lower degree lying on complete intersection surfaces.

\bibliography{curvesgenus}
\bibliographystyle{halpha}

{\tt School of Mathematics and Maxwell Institute, The University of Edinburgh,  
James Clerk Maxwell Building,
The King's Buildings, Mayfield Road,
Edinburgh, EH9 3JZ, United Kingdom \\
r.tramel@ed.ac.uk \\
www.maths.ed.ac.uk/$\sim$s1261516/ }
\end{document}